\documentclass[reqno,12pt]{amsart}
\usepackage{amsmath,amsthm,amsfonts,amssymb,amscd, bbm, xcolor}
\usepackage{enumerate, ulem}
\usepackage{geometry}
\usepackage{enumitem}
\usepackage{graphicx}
\theoremstyle{plain}
\newtheorem{maintheorem}{Theorem}

\newcommand{\vertiii}[1]{{\left\vert\kern-0.25ex\left\vert\kern-0.25ex\left\vert #1 
		\right\vert\kern-0.25ex\right\vert\kern-0.25ex\right\vert}}

\newtheorem{maincorollary}[maintheorem]{Corollary}

\newtheorem*{theorem*}{Theorem}

\newtheorem{theorem}{Theorem}[section]

\newtheorem{corollary}[theorem]{Corollary}
\newtheorem{lemma}[theorem]{Lemma}
\newtheorem{remark}[theorem]{Remark}
\newtheorem*{remark*}{Remark}

\newtheorem{proposition}[theorem]{Proposition}

\newcommand{\N}{\ensuremath{\mathbb{N}}}

\newcommand{\R}{\ensuremath{\mathbb{R}}}

\newcommand{\Z}{\ensuremath{\mathbb{Z}}}

\newcommand{\cA}{{\mathcal A}}
\newcommand{\cB}{{\mathcal B}}

\newcommand{\cH}{{\mathcal H}}

\newcommand{\cL}{{\mathcal L}}

\newcommand{\vep}{\varepsilon}

\newcommand{\diam}{\operatorname{diam}}

\newcommand{\field}[1]{\mathbb{#1}}
\newcommand{\real}{\field{R}}

\newcommand{\al} {\alpha}

\renewcommand{\field}[1]{\mathbb{#1}}

\newcommand{\vr}{\varphi}

\begin{document}
\title[Regularity, linear response formula and differentiability...]{Regularity, linear response formula and differentiability of the free energy for non-uniformly expanding local homeomorphisms}
	
\date{}
\author{Carlos Bocker}
\address[Carlos Bocker]{Department of Mathematics, UFPB\\ Jo\~ao Pessoa-PB, Brazil}
\email{cbocker@gmail.com}

\author{Ricardo Bortolotti}
\address[Ricardo Bortolotti]{Department of Mathematics, UFPE\\ Recife-PE, Brazil}
\email{ricardo.bortolotti@ufpe.br}

\author{Armando Castro}
\address[Armando Castro]{Department of Mathematics, UFBA\\ Salvador-BA, Brazil}
\email{armandomath@gmail.com}

\author{Sávio Santana}
\address[Sávio Santana]{Department of Mathematics, UFBA\\ Salvador-BA, Brazil}
\email{santana.savio@ufba.br}

\begin{abstract}
We study equilibrium states for an open class of non-uniformly expanding local homeomorphisms defined by a 
mild condition such that for some iterate each point admits at least one contracting inverse branch. 
We prove the existence and uniqueness of equilibrium states and the differentiability of statistical quantities (such as the equilibrium states and the free energy function) with respect to the dynamical system.
\end{abstract}

\maketitle



\section{Introduction}

Let $M$ be a compact metric space, given a continuous map $f:M \to M$ and a
continuous potential $\phi:M \to \R$, the variational principle asserts
that
$$\displaystyle P_{top}(f, \phi) = \sup \Big\{ h_\mu(f) +  \int \phi d\mu \Big\},$$
where $P_{top}(f,\phi)$ denotes the topological pressure of $f$ with respect to $\phi$, $h_\mu(f)$ denotes the metric entropy, and the supremum  is taken over all the $f$-invariant probability measures. An \textit{equilibrium state} for $f$ with respect to $\phi$ is an
invariant measure that attains the supremum in the right hand side above.

Chaotic dynamical systems usually admit infinitely many invariant measures and the equilibrium states are among the most important invariant measures for such systems. Their concept was brought from statistical mechanics to hyperbolic dynamical systems, due to the existence of Markov partitions the notions of Gibbs measures and equilibrium states were translated to ergodic theory.

For transitive uniformly expanding and uniformly hyperbolic maps, equilibrium states always exist and they are unique when the potential is Hölder continuous (\cite{Ru89, Bow75}).

Also inspired from physics, one can ask whether the equilibrium state is differentiable with respect to the dynamical system, this property has been refereed as linear response formula. For uniformly expanding maps, when the potential is $C^r$-smooth, the thermodynamical quantities are also differentiable and varies differentiably  with respect to the dynamics (\cite{Ru89, GP17, BC19}). Linear response formulas have been obtained for more general dynamical systems, such as one-dimensional piecewise expanding and quadratic maps (\cite{BS08, BS09, BS12, Ru05}) and intermittent maps (\cite{BSa16, BTo16, K16, Sedro}).

Many important classes of dynamical systems are not uniformly expanding but present some kind of expanding behaviour, such as the Manneville-Pomeau map. Although they are not uniformly expanding, these class of maps present some non-uniform expansion.
Several developments were made to study non-uniformly expanding dynamical systems  
(\cite{OV, VV10, BCV16}). The techniques used in these works were based in a spectral analysis of the transfer operator and allowed to obtain similar results of existence and uniqueness of equilibrium states and their differentiable dependence (linear response formulas).

The open class of non-uniformly expanding maps usually considered in \cite{OV, VV10, CV2013, BCV16} is formulated considering a ``bad set" of points where the dynamics does not expands and supposing that this set is covered by few domains of injectivity (less than the degree).

In the cases above it was noticed that the potential is hyperbolic \cite{VV10}. In recent works \cite{RV, ARS2021}, it was considered hyperbolic potentials instead of asking directly that the dynamic is non-uniformly expanding. In \cite{ARS2021} they notice that if the map is also strong topologically mixing then every point admits some inverse branch that is a contraction.

In this work we formulate a simpler condition that defines a class of non-uniformly expanding maps without mentioning the existence of ``bad sets", we just ask that every point admits an inverse branch that is a contraction (Condition (C) in Section 2).
We do not ask that the map is transitive, actually, this class of maps includes systems where the equilibrium state is supported in a proper repeller.

We prove the existence and uniqueness of equilibrium states for the class of dynamical systems defined by this simpler condition. We also prove statistical properties such as exponential decay of correlations and  Central Limit Theorem for these equilibrium states, and  the differentiable dependence of certain statistical quantities (such as the free energy function and the rate functions for a large deviation principle) with respect to the dynamical systems.

This paper is organized as follows. In Section 2, we describe the class of dynamical systems that we consider and state the main results. In Section 3 we prove the main proposition of this work, that connects the simpler condition of this work with previous known conditions of non-uniform expansion involving "good" and "bad" regions.
In Section 4 we give the proof of the theorems. Finally, in Section 5 we discuss examples of dynamical systems satisfying this condition.

\section{Statements of the results}

Let $M$ be a compact and connected Riemannian manifold with distance $d$.
Let $f:M \to M$ be a \textit{local homeomorphism} and assume that there exists a continuous function
$x\mapsto L(x)$ such that for every $x\in M$ there is a
neighborhood $U_x$ of $x$ so that $f_x : U_x \to f(U_x)$, defined by $f_x(y)=f(y)$, is
invertible and
\begin{equation}
d(f_x^{-1}(y),f_x^{-1}(z))
\leq L(x) \;d(y,z), \quad \forall y,z\in f(U_x).
\end{equation}

Let $L_f=\sup_{w\in M} L(w)$, we know that there exists $\delta_0>0$ depending only on $L_f$ and $\deg(f)$ such that for all $x$, $f^{-1}(B(x,\delta_0))$ is a disjoint union of open sets $U_1, \cdots, U_{\deg(f)}$ such that $f_i=f_{\restriction U_i}$ is a homeomorphism and
\begin{equation}\label{L}
d(f_i^{-1}(y),f_i^{-1}(z))
\leq L_f \;d(y,z), \quad \forall y,z\in B(x,\delta_0).
\end{equation} 

Writing $L_k(x) =  L(x) L(f(x)) \cdots   L( f^{k-1}(x))$ and $\delta_k = \delta_0 L_f^{-k}$, we have that the iterates $f^k$ are also local homeomorphisms whose inverse branches satisfy \eqref{L} with $L_{f^k} \leq  \sup_{w\in M}L_k(w)$ and $\delta_k$ instead of $L_f$ and $\delta_0$, respectively.

Given numbers $0<\gamma<1$, $L \geq 1$ and integers $N, G \geq 2$,  denote by $\cH_{\gamma, L, N, G}$ the set of local homeomorphisms $f:M\to M$ such that  $\deg(f)\leq G$, $L_{f^N}\leq L$ and satisfy the following condition:


\begin{itemize}\label{properties of f}
	\item[\textbf{(C)}] For each $x\in M$ there is $y\in f^{-N}(x)$ such that $L_N(y)  < \gamma$.\label{propety b}
\end{itemize}

\vspace{0.3cm}

Condition (C) means that for some iterate $N$ every point admits at least one inverse branch that is a contraction.

Denote by $\cH^\alpha_{\gamma, L, N, G}$, $0<\alpha<1$,  the set of maps in $\cH_{\gamma, L, N, G}$ that are  $\alpha$-H\"older continuous, by $\cH^\alpha_{\gamma, L, N, G,\delta_0}$ the set of maps that also satisfy \eqref{L} for the same constant $\delta_0$. And for $r\geq 1,\,r\in\Z$, denote by $\cH^r_{\gamma, L, N, G}$ the set of maps in $\cH^\alpha_{\gamma, L, N, G}$  that are $C^r$ maps.

The \textit{Ruelle-Perron-Fr\"obenius transfer operator} $\cL_{f,\phi}$ associated to 
$f:M\to M$ and $\phi:M\to\real$ is the linear operator defined on a Banach space $X \subset C^0(M,\mathbb R)$ 
of continuous functions $\vr:M\to\real$ defined by
\begin{equation}
\cL_{(f,\phi)} \vr(x) = \sum_{f(y)=x} e^{\phi(y)}\vr(y).
\end{equation}

Since $f$ is a local homeomorphism it is clear that $\cL_{(f,\phi)} \vr$ is continuous for every continuous $\vr$,
 when $f \in C^\alpha(M,M)$ the operator $\cL_{(f,\phi)}$ preserves the Banach space $C^{\alpha}(M,\mathbb R)$, $0<\alpha<1$, of H\"older
 continuous observables. And when $f$ is a $C^r$-local diffeomorphism, $\cL_{(f,\phi)}$ preserves the Banach  space $C^{r}(M,\mathbb R)\subset C^0(M,\mathbb R)$, $r\ge 1$.

We say that the Ruelle-Perron-Frobenius operator $\cL_{(f,\phi)}$ acting on a Banach space $X$ has the
\textit{spectral gap property} if there exists a decomposition of its spectrum 
$\sigma(\cL_{(f,\phi)})\subset \mathbb C$ as follows: $\sigma(\cL_{(f,\phi)})=\{\lambda_1\}\cup \Sigma_1$ where
$\lambda_1$ is a leading eigenvalue for $\cL_{(f,\phi)}$ with one-dimensional associated eigenspace
and $\Sigma_1 \subsetneq \{ z\in \mathbb C : |z|<\lambda \}$ for some $0< \lambda < \lambda_1$.

An $f$-invariant probability $\mu$ has \textit{exponential decay of
	correlations} for H\"older continuous observables if
there exists some constants $0<\tau<1$ such that
for all $\varphi\in L^1(\mu), \psi\in C^{\alpha}(M)$ there exists $K(\varphi,\psi)>0$ satisfying
\begin{equation*}
	\left|\int_M (\varphi\circ f^n) \psi d\mu - \int_M \varphi d\mu\int_M \psi d\mu\right|
	\leq K(\varphi,\psi)\cdot\tau^n,
	\quad \text{for every $n\ge 1$}.
\end{equation*}

We say that $\mu$ is \textit{exact} if all observable $\varphi$ which is measurable with respect to the $\sigma$-algebra $\cA=\cap_{n=1}^{\infty} f^{-n}(\cB)$ is constant ($\cB$ is the Borelian $\sigma$-algebra). Equivalently $\mu$ is exact if $\varphi=\varphi_n\circ f^n$ for some $\varphi_n$ in the set of observables and all $n\in\N$ implies that $\varphi$ is constant.

\subsection{\bf H\"older continuous setting}

An observable $g: M\to \mathbb R$ is $\al$-H\"older continuous if the H\"older constant  
$$
|g|_\al
=\sup_{x\neq y} \frac{|g(x)-g(y)|}{d(x,y)^\al}
$$
is finite. As usual, we endow the space $C^\al(M, \mathbb R)$ of H\"older continuous observables with
the norm $\|\cdot\|_\al=\|\cdot \|_0+|\cdot|_\al$.
Consider the set $P^{\alpha}(\vep)$ of $\alpha$-H\"older
continuous   potentials $\phi:M \to \mathbb R$ such that 
\begin{equation}\label{potential}
	\sup\phi-\inf\phi<\vep \quad \text{ and } \quad |e^{\phi}|_{\alpha}  <\vep \;e^{\inf \phi}
\end{equation}
where $\vep$ is chosen small enough (satisfying equation \eqref{eq. castro varandas 1} in Section 4).

\begin{maintheorem}\label{thm.spectralgap}(Equilibrium states for H\"older potentials)
	Given $\gamma\in(0,1)$ and integers $N, G \geq 2$, there are  $L>1$ and $\vep>0$ such that for all $f\in \cH^\alpha_{\gamma, L, N, G}$ and for all potential $\phi\in P^\alpha(\vep)$ the transfer operator $\cL_{f,\phi}$ has the spectral gap property in the space of $C^\alpha$ observables.
	Denoting by $\lambda_{f,\phi}=\rho(\mathcal{L}_{f,\phi})$ the spectral radius of $\mathcal{L}_{f,\phi}$, the following properties hold:
	\begin{enumerate}
		\item There exists a unique probability measure $\nu_{f,\phi}$, called reference measure, such that 
		$\mathcal{L}_{f,\phi}^*\nu_{f,\phi}=\lambda_{f,\phi} \nu_{f,\phi};$
		\item There is a unique positive function $h_{f,\phi}: M \to \R$ of class $C^\alpha$ bounded from zero and infinity such that $\mathcal{L}_{f,\phi}h_{f,\phi}=\lambda_{f,\phi} h_{f,\phi}$
		with $\int h_{f,\phi}\, d\nu_{f,\phi}=1$.
		\item The $f$-invariant measure $\mu_{f,\phi}=h_{f,\phi}\nu_{f,\phi}$ is the unique equilibrium state of $(f,\phi)$, it has exponential decay of correlation for $C^\alpha$ observables and is exact.
		\item The topological pressure satisfies $P_{top}(f,\phi) = \log \lambda_{(f,\phi)}$.
	\end{enumerate}
	Moreover, the following maps are continuous:
	\begin{enumerate}[label=(\alph*)]
		\item $(f,\phi)\in \cH^\alpha_{\gamma, L, N, G,\delta_0} \times P^\alpha(\vep) \mapsto \lambda_{f,\phi}\in \R$.
		\item $(f,\phi)\in \cH^\alpha_{\gamma, L, N, G,\delta_0} \times P^\alpha(\vep) \mapsto P_{top}(f,\phi)\in \R$.
		\item $(f,\phi) \in \cH^\alpha_{\gamma, L, N, G,\delta_0} \times P^\alpha(\vep) \mapsto h_{f,\phi}\in C^0$;
		\item $(f,\phi)\in \cH^\alpha_{\gamma, L, N, G,\delta_0} \times P^\alpha(\vep) \mapsto \nu_{f,\phi}\in C^{\alpha}(M,\R)^*$.
		\item $(f,\phi)\in \cH^\alpha_{\gamma, L, N, G,\delta_0} \times P^\alpha(\vep) \mapsto \mu_{f,\phi}\in C^{\alpha}(M,\R)^*$.
	\end{enumerate}
And they are analytic with respect to $\phi$.
\end{maintheorem}

\subsection{\bf $C^{r+\alpha}$-smooth setting}

If we deal with differentiable local diffeomorphisms $f$ and potentials $\phi$, then we can conclude that the maps and objects of Theorem A are even more regular.

Let $r \geq 1$ and $0<\alpha<1$ be, we define $\mathcal{H}^{r+\alpha}_{\gamma,N,L,G}$ by the set of $C^r$ maps whose their $r$-derivatives are $\alpha$-H\"older continuous and satisfy the condition \textbf{(C')} below:

\begin{center}
	\textbf{\textbf{(C')}} - For each $x\in M$ there is $y\in f^{-N}(x)$ such that $\|Df^{N}(y)^{-1}\|< \gamma.$
\end{center}

Consider the set $P^r(\vep,\vep')$ of $C^r$ potentials $\phi:M\to\R$ which satisfy the condition
\begin{equation}
	\sup \phi-\inf \phi=\epsilon_\phi<\vep \quad \text{  and } \quad  \|D\phi\|_{C^{r-1}}=\vep'_\phi<\vep'.
\end{equation}
for some $\vep>0$ and $\vep'>0$  chosen small enough (in Section 4),  depending only on $L$, $N$, $G$, $r$ and $\gamma$.

\begin{remark*}
	The sets $\cH^{r+\alpha}_{\gamma, L, N, G}$ and $P^r(\vep,\vep')$ are open sets with respect to the $C^{r+\alpha}$ and $C^r$ topologies, respectively. For every  $f \in \cH^{r+\alpha}_{\gamma, L, N, G}$ there exists a uniform $\delta_0$ such that $\eqref{L}$ is valid in a $C^1$-neighborhood of $f$.
\end{remark*}

\begin{maintheorem}\label{thm. spectralgapCr}(Spectral gap for smooth observables)
	Given $\gamma\in(0,1)$ and integers $r \geq 1$, $N, G \geq 2$, there are $L>1$, $\epsilon>0$ and $\vep'>0$ such that for all $f\in \cH^{r+\alpha}_{\gamma, L, N, G}$ and all potential $\phi\in P^r(\vep,\vep')$ the transfer operator $\cL_{f,\phi}$ has the spectral gap property in the space of $C^r$ observables.
	
	As a consequence, the unique eigenfunction $h_{f,\phi}$  associated to the eigenvalue  $\lambda_{f,\phi}$ is of class $C^r$ and the following maps are of class $C^{r-1}$:

	\begin{enumerate}[label=(\alph*)]
		\item $(f,\phi)\in \cH^{r+\alpha}_{\gamma, L, N, G} \times P^r(\vep,\vep') \mapsto \lambda_{f,\phi}\in \R$.
		\item $(f,\phi)\in \cH^{r+\alpha}_{\gamma, L, N, G}  \times P^r(\vep,\vep') \mapsto P_{top}(f,\phi)\in \R$.
		\item $(f,\phi) \in \cH^{r+\alpha}_{\gamma, L, N, G}  \times P^r(\vep,\vep') \mapsto h_{f,\phi}\in C^0$;
		\item $(f,\phi)\in \cH^{r+\alpha}_{\gamma, L, N, G}  \times P^r(\vep,\vep') \mapsto \nu_{f,\phi}\in C^{r}(M,\R)^*$.
		\item $(f,\phi)\in \cH^{r+\alpha}_{\gamma, L, N, G}  \times P^r(\vep,\vep') \mapsto \mu_{f,\phi}\in C^r(M,\R)^*$.
	\end{enumerate}
	Moreover these maps  are analytic with respect to $\phi$.
\end{maintheorem}

\subsection{Statistical consequences}

As consequence of the spectral gap property, we can deduce the regularity of some statistical quantities involving the equilibrium state that we describe in the following.

In all the three Corollaries below, we assume that 
$(f,\phi)\in \cH^\alpha_{\gamma, L, N, G} \times P^\alpha(\vep)$ satisfy the assumptions of Theorem A, when $0 < \alpha \leq 1$, and that 
$(f,\phi)\in \cH^{r+\alpha}_{\gamma, L, N, G} \times P^r(\vep,\vep')$ satisfy the assumptions of Theorem B, when $r \geq 1$.

\subsubsection{Central Limit Theorem}

The property of exponential decay of correlations implies a Central Limit Theorem.
\begin{maincorollary}[Central Limit Theorem]\label{c.CLT}
	Let $\varphi$ be a H\"older
	continuous function and  
	$$
	\sigma_{f,\varphi,\mu}^2 = \sigma_\varphi^2 :=\int v^2 d\mu + 2\sum\limits_{j=1}^{\infty} v\cdot (v\circ f^j) \, d\mu,
	\quad \text{ where } \quad v=\varphi-\int \varphi \, d\mu.
	$$
	Then:
	\begin{itemize}
		\item[(i)] $\sigma_\varphi<\infty$ and $\sigma_\varphi=0$ if and only if $\varphi=u\circ f - u$ for some
		$u \in L^1(\mu)$. Furthermore, if $\sigma_\varphi>0$ then the following convergence on distribution
		\begin{equation*}
			\mu\left(x\in M: \frac{1}{\sqrt{n}}\sum\limits_{j=0}^{n-1}
			\left(\varphi(f^j(x))-\int \varphi d\mu\right)\in A\right)\to \frac{1}{\sigma_\varphi\sqrt{2\pi}}
			\int_A e^{-\frac{t^2}{2\sigma_\varphi^2}} dt,
		\end{equation*}
		holds as $n\to\infty$ for every interval $A\subset\real$.
	  
	\item[(ii)] For $r\geq 2$, the average $$\cH^{r+\alpha}(\gamma, L, N, G) \times P^r(\vep,\vep') \times C^r(M,\R) \ni 
	(f, \phi, \varphi) \to m_{\mu_{(f,\phi)}}(\varphi) := \int \varphi d\mu_{(f,\phi)} \in \R$$
	and  the variance 
	$$\cH^{r+\alpha}(\gamma, L, N, G) \times P^r(\vep,\vep') \times C^r(M,\R) \ni 
	(f, \phi, \varphi) \to \sigma^2_{f,\varphi,\mu_{(f,\phi)}} \in \R$$
	are $C^{r-1}$-maps and 	are analytic with respect only to $(\phi,\varphi)$. 
\end{itemize}
\end{maincorollary}

\subsubsection{Free energy}
Given an observable $\psi: M \to \R$ and $t \in \R$, the free energy $\mathcal{E}_{f,\phi,\psi}$ is defined by: 
$$\mathcal{E}_{f,\phi,\psi}(t) := \underset{n \to \infty}{\limsup} \dfrac{1}{n} \log \int e^{t S_n \psi} d \mu_{f,\phi},$$
where $S_n\psi = \sum_{j=0}^{n-1} \psi\circ f^j$ is the Birkhoff sum. Then 
\begin{maincorollary}[Differentiability of the Free Energy]\label{C.free}
	For any  $\alpha$-H\"older observable $\psi: M \to \R$ there exists $t_{\phi,\psi}>0$ such that if $|t| \leq t_{\phi,\psi} $, then  $(f,\phi+t\psi)\in \cH^\alpha_{\gamma, L, N, G} \times P^\alpha(\vep)$, the free energy is well defined and satisfy:
	
	$$\mathcal{E}_{f,\phi,\psi} (t) := \underset{n \to \infty}{\limsup} \dfrac{1}{n} \log \int e^{t S_n \psi} d \mu_{f,\phi} = \log \lambda_{(f, \phi + t \psi)} - \log \lambda_{(f, \phi )} 	 $$
	As a consequence, $(t,\phi,\psi) \mapsto  \mathcal{E}_{f,\phi,\psi}(t) $ is analytic.
	It is valid that if $\psi$ is cohomologous to a constant then $t \mapsto \mathcal{E}_{f,\phi,\psi}$ is affine and, otherwise, it is strictly convex in $[-t_{\phi,\psi}, t_{\phi,\psi}]$. 
	
	If $(f,\phi)\in \cH^{r+\alpha}_{\gamma, L, N, G} \times P^r(\vep,\vep')$ satisfy the assumption of Theorem B, then $\mathcal{E}_{f,\phi,\psi} (t)$ is $C^{r-1}$.	
\end{maincorollary}

\begin{remark}
	We remark that one novelty of the Theorem above is the differentiability of higher order ($C^{r-1}$) with respect to the dynamical system. In previous papers, such as \cite[Theorem E]{BCV16}, it was only proved the $C^1$-differentiability with respect to the dynamical system.
\end{remark}

\subsubsection{Legendre transform and large deviations principle}

	When $\psi$ is not cohomologous to a constant, then the convexity of $t \in [-t_{\phi,\psi}, t_{\phi,\psi}] \mapsto \mathcal{E}_{f,\phi,\psi}(t)$ implies that it is well defined the local Legendre transform $I_{f,\phi,\psi}$ given by 
	$$I_{f,\phi,\psi}(s) = \underset{ -t_{\phi,\psi} \leq t \leq t_{\phi,\psi}  }{\sup} (st - \mathcal{E}_{f, \phi, \psi}(t) ).$$
	
	The differentiability that we obtained above gives a large deviation principle with differentiable rate function, as below

\begin{maincorollary}[Large Deviation Principle with Differentiable Rate]\label{cor.I}
Let $V$ be a compact metric space, $\{(f_v, \phi_v, \psi_v)\}_{v\in V}$ 
a continuous parameterized injective family of maps in 
$\cH^{r+\alpha}_{\gamma, L, N, G} \times P^r(\vep,\vep')\times C^r(M,\R)$, $r \geq 2$.
If the observable $\psi_{v*} \in C^r$ is not cohomologous to a constant then there exist an open neighborhood $U$ of $v_*$ and an open interval $J$ such that: for all $v\in U$ and $[a,b] \subset J$ it is valid:
$$\underset{n \to \infty}{\limsup} \dfrac{1}{n} \log \mu_{f_v, \phi_v }\Big( x \in M, \dfrac{1}{n} S_n \psi_v (x) \in [a,b]  \Big)  \leq - \inf_{s \in [a,b]} I_{f_v, \phi_v, \psi_v}(s)$$
and  
$$\underset{n \to \infty}{\liminf} \dfrac{1}{n} \log \mu_{f_v, \phi_v }\Big( x \in M, \dfrac{1}{n} S_n \psi_v (x) \in (a,b)  \Big)  \geq - \inf_{s \in [a,b]} I_{f_v, \phi_v, \psi_v}(s).$$

And the rate function $(s,v) \in J \times U \mapsto I_{f_v, \phi_v, \psi_v}(s)$ 
 is of class $C^{r-1}$.
\end{maincorollary}

We notice that large deviations principles for non-uniformly expanding maps are already known (\cite{araujo.pacifico, BCV16}), the novelty is the $C^{r-1}$-differentiability.

\section{The Main Proposition}
The following proposition is the main result that allows to connect our setting to the ones found in the literature (such as \cite{CV2013} and \cite{BCV16}).
In Section 4 we will see that, although our hypothesis are weaker than the one in those works, they imply the same type of results.

\begin{proposition}[Main  Proposition]\label{key proposition}
	Suppose that $L\geq 1$ is sufficiently close to 1 and $f \in \cH^{\alpha}_{\gamma, L, N, G}$, then for some integer $m_0 \geq 1$  and every $m \geq m_0$ there exists an open region $\mathcal{A}_m \subset M$ and a constant $\gamma_m \in (0,1)$ satisfying:
	
	\textbf{(H1)} $L_m(x) \leq L^m$ for every $x \in \mathcal{A}_m$ and $L_m(x) < \gamma_m$ for all $x \notin \mathcal{A}_m$.
	
	\textbf{(H2)} There exists a finite covering $\mathcal{U}_m$ of $M$ by open domains of injectivity for $f^m$ such that $\mathcal{A}_m$ can be covered by $q_m < \deg(f^m)$ sets of this covering.
\end{proposition}

Let us state first an useful lemma for the proof of Proposition~\ref{key proposition}.

\begin{lemma}
	There exists some constant $C>0$ such that for every $r \in (0,1)$ it is possible to cover $M$ with at most $C r^{- \dim M}$ balls of radius $r$.
\end{lemma}
\begin{proof}
	If $M$ is a compact subset of $\R^k$, for some $k$, and $d$ is the distance induced by the Euclidean norm then the result is valid for $C=(\sqrt{2}\diam{M})^{\dim{M}}$. If we consider another norm, we obtain the same result with some constant $\tilde C$ slightly worse, since every norm in $\R^k$ is equivalent.
	
	Denote by $B(x,a)$ the Euclidean ball of radius $a$ centered in $x$.
	Consider finite charts $\varphi_i: V_i \subset M \to B(0,2) \subset \R^k$ such that $M$ is covered by $\{  \varphi_i^{-1}(B(0,1))\}_i$  and such that $1/2 < \|D\varphi_i^{-1}\| \leq \|D\varphi_i\| \leq 2$. Every set $ \varphi_i^{-1}(B(x,r/2))$ is contained in a single ball of radius $r$, so we can cover $B(1)$ by $C r^{- \dim M}$ balls of radius $r/2$ and consider their inverse images by the charts, which is a covering of $M$ as we want.	 
\end{proof}

In the proof of Proposition \ref{key proposition} we say that a inverse branch $f^{-k}$ of $f^k$, defined in a ball of radius $\delta_k$, contracts at rate $\gamma$ if
$$d(f^{-k}(y), f^{-k}(z) ) \leq \gamma d(y,z), \quad \forall y, z \in B(x,\delta_k).$$

Now we proceed to the proof of Proposition \ref{key proposition}.

\begin{proof}[Proof of Proposition \ref{key proposition}]

Consider $L$ sufficiently close to $1$  so that we can take $\ell = \ell(L)$ large satisfying:
\begin{align}\label{estimate.1}
	&\tilde\gamma = \gamma L^{N(\ell-1)}  < 1  \quad\quad \quad  \text{  and  }  \\  
	&e^{\frac{1}{\ell}} \ell^{\frac{1}{\ell}}L^{N \dim M} < \dfrac{G^N}{G^N-1}. \label{estimate.2}
\end{align}

	Fix $\delta$ such that for every ball of radius $\delta$ are well defined the $\deg(f)$ inverse branches of $f$.
	Then for $\displaystyle \delta_j= \delta L^{-j}$, for every ball of radius $\delta_j$ are well defined the $\deg(f)^j$ inverse branches of $f^j$.
	
	Consider an open cover $\mathcal{U}_j = \{U_i\}_{i=1,\cdots D_j}$ of $M$ by balls of radius at most $\delta_j$, formed by at most $D_j=C \delta_j^{-j \dim M}$ balls of radius $\delta^j$.
	
Due to the choice of $L$ we have that	
for every $U_i$, the inverse branches of $f^{\ell N}$ at $U_i$ are composition of inverse branches of $f^{N}$, if one of these contracts at rate $\gamma$ then the composition contracts $\tilde \gamma$. So there are at most $(\deg(f)^N-1)^\ell$ inverse branches of $f^{\ell N}$ at $U_i$ that do not contracts at rate $\tilde \gamma$.
		
	Now, for every $k \geq 1$, the $\deg(f)^{\ell kN}$ inverse branches of $f^{\ell kN}$ at $U_i$ are composition of $\ell k$ inverse branches of $f^N$. If $k$ of these inverse branches contracts at rate $\tilde \gamma$, then the composition contracts at rate  $\tilde \gamma^k$.	
	We can count how many inverse branches of $f^{\ell kN}$  at $U_i$ do not contracts $\tilde \gamma$: they are composition of $\ell k$ inverse branches of $f^N$ where at most $k-1$ of them contract at rate $\tilde \gamma$; so they are at most:
	$$
	B_k=\sum_{j=0}^{k-1}{{\ell k}\choose{j}} (G^N-1)^{\ell k - j}
	$$
	
	Since we have covered the manifold $M$ by $D_k$ open sets of the open cover $\mathcal{U}_k$, we can define the set $A_k$ as the union of the open sets that are the inverse images of $U_i$ by all the inverse branches of $f^{\ell kN}$, $i=1,\cdots, D_k$ that do not contracts at rate $\tilde \gamma^k$. If $q_k$ is the cardinality of this cover, we have 
	$$q_k \leq D_k \cdot B_k$$
	
	The Proposition will follow for $m=\ell k N$ if 
	\begin{equation}\label{equation}
		q_k < G^{\ell kN}
	\end{equation}
 since we can consider $\gamma_m = \tilde \gamma^k$, $L_m= L^{\ell k N}$,  $\mathcal{A}_m= A_k$ and $q_m = q_k$.
	
	But \eqref{equation} this is true for $k$ sufficiently large, this can be verified using the relation for $0 \leq j \leq k$: $$\displaystyle  {{\ell k}\choose{j}} \leq {{\ell k}\choose{k}} \leq \big( \dfrac{e\ell}{k}\big)^k \leq (e\ell)^k.$$ Actually:
	\begin{align*}
	A_k &\leq D_k (G^N-1)^{\ell k }k  (e\ell)^{k} \\
	&\leq C k (e\ell)^k L^{kN \dim M} (G^N-1)^{\ell  k }   \\
	&= C k \Big[ L^{N \dim M} (G^N-1) (e \ell)^{\frac{1}{\ell}}   \Big]^{ \ell k}.
	\end{align*}
By \eqref{estimate.1}, it follows that this number is smaller than $G^{\ell kN}$  for all $k$ sufficiently large.

Finally, the result is valid for every $m$ large enough. Consider $k_0$  large enough so that $\tilde{\tilde{\gamma}} := \tilde \gamma^{k_0} L^{\ell N} < 1$, then  $m > m_0 := \ell k_0 N$. 

	If $m > m_0$ is not multiple of $\ell N$, we can write $ m = \ell N k + r$, with $0\leq r < \ell N$, so every inverse branch of $f^{ m}$ is the composition of $k$ inverse branches of $f^{\ell N}$ with $r$ inverse branches of $f$. If those $k$ inverse branches contracts at rate $\tilde\gamma$, then the composition contracts at rate $\tilde{\tilde{\gamma}}$.
	
	This allows to construct $\mathcal{A}_{ m }$ identically: define the set $A_k$ as the union of the open sets that are the inverse images of $U_i \in \mathcal{U}_k$ by all the inverse branches of $f^{ m}$, $i=1,\cdots, D_k$ that do not contracts at rate $\tilde{\tilde{ \gamma}}$. If $Q_{m}$ is the cardinality of this cover, we have 
$Q_{ m} \leq  q_k C_1    $ for some constant $C_1$ that do not depends on $k$ and $G^{m} \geq G^{ lk N} C_2 $ for another constant $C_2$. Then the cardinality of $\mathcal{A}_{\tilde m }$ is smaller than $G^{\tilde m}$ if $k$ is large enough.

\end{proof}

\section{Proof of the Theorems}

\begin{proof}[Proof of Theorem~\ref{thm.spectralgap}]
For $(f,\phi)\in \cH^\alpha_{\gamma, L, N, G} \times P^\alpha(\vep)$,
consider some $m$ given by Proposition~\ref{key proposition} and the corresponding 
$L_m=L^m$, $\gamma_m \in (0,1)$ and $q_m<\deg(f^m)$. Then we consider 
$L$ sufficiently close to 1 and $\vep>0$ small enough  such that:
\begin{equation}\label{eq. castro varandas 1}
(1+\vep)e^\vep \Big(\frac{\left( \deg(f^m) - q \right) \gamma_m^{\alpha} + qL_m^\alpha \left[1 + (L_m - 1)^\alpha \right]}{\deg(f^m)}\Big)  + 2\vep L_m^\alpha \text{diam}(M)^\alpha < 1.
\end{equation}

This allows to apply \cite[Theorem A]{CV2013} 
to the map $g=f^m$ and the potential $\tilde \phi=\sum_{j=0}^{m-1}\phi\circ f^j$ to conclude that 
the transfer operator $\cL_{g,\tilde\phi}$ has the spectral gap property in the space of $C^\alpha$-observables.
If  $\hat\lambda_{g,\tilde\phi}=\rho(\mathcal{L}_{g,\tilde\phi})$  denotes the spectral radius of $\mathcal{L}_{g,\tilde\phi}$ it follows  that
	there exists a unique probability measure $\hat\nu_{g,\tilde\phi}$ 
	 such that 
	$\mathcal{L}_{g,\tilde\phi}^*\hat\nu_{g,\tilde\phi}=\hat\lambda_{g,\tilde\phi} \hat\nu_{g,\tilde\phi}$,  
	there is a unique positive function $\hat h_{g,\tilde\phi}: M \to \R$ of class $C^\alpha$ bounded from zero and infinity such that $\mathcal{L}_{g,\tilde\phi}\hat h_{g,\tilde\phi}=\hat\lambda_{g,\tilde\phi}\hat h_{g,\tilde\phi}$
	with $\int\hat h_{g,\tilde\phi}\, d\hat\nu_{g,\tilde\phi}=1$, 
	the $g$-invariant measure $\hat\mu_{g,\tilde\phi}=\hat h_{g,\tilde\phi}\hat\nu_{g,\tilde\phi}$ is the unique equilibrium state of $(g,\tilde\phi)$, it has exponential decay of correlation for $C^\alpha$ observables and is exact.

Since $\cL_{g,\tilde\phi}=\cL^m_{f,\phi}$, the transfer operator $\cL_{f,\phi}$ also has the spectral gap property in the space of $C^\alpha$ observables with
$\lambda_{f,\phi}=\sqrt[m]{\hat\lambda_{g,\tilde\phi}}=\rho(\mathcal{L}_{f,\phi})$ denoting the spectral radius of  $\mathcal{L}_{f,\phi}$.
Since we also have $(\mathcal{L}^*_{f,\phi})^m\hat\nu_{g,\tilde\phi}=\hat\lambda_{g,\tilde\phi} \hat\nu_{g,\tilde\phi}$ and $\mathcal{L}_{f,\phi}^m\hat h_{g,\tilde\phi}=\hat\lambda_{g,\tilde\phi}\hat h_{g,\tilde\phi}$, by uniqueness it follows that  $\nu_{f,\phi}=\hat\nu_{g,\tilde\phi}$ and  $h_{f,\phi}=\hat h_{g,\tilde\phi}$ are the unique eigenvectors of $\cL_{f,\phi}^*$ and $\cL_{f,\phi}$ associated to the common eigenvalue $\lambda_{f,\phi}=\sqrt[m]{\hat\lambda_{g,\tilde\phi}}$, respectively, and satisfy $\int h_{f,\phi}\, d\nu_{f,\phi}=1$.

Since $P_{top}(g,\tilde\phi)=m P_{top}(f,\phi)$, defining $\mu_{f,\phi}=\hat\mu_{g,\tilde\phi}=\hat h_{g,\tilde\phi}\hat\nu_{g,\tilde\phi}$ and   $P_{\mu_{f,\phi}}(g,\tilde\phi)=m P_{\mu_{f,\phi}}(f,\phi)$, we conclude that $\mu_{f,\phi}$ is the unique equilibrium state of $(f,\phi)$. Exponential decay and exactness of $\mu_{f,\phi}$ follow immediately from the definition, since it holds for $f^m$.

Now the continuous variation of the  maps
\begin{enumerate}[label=(\alph*)]
	\item $(f,\phi)\in \cH^\alpha_{\gamma, L, N, G} \times P^\alpha(\vep) \mapsto \lambda_{f,\phi}\in \R$.
	\item $(f,\phi) \in \cH^\alpha_{\gamma, L, N, G} \times P^\alpha(\vep) \mapsto h_{f,\phi}\in C^0$;
	\item $(f,\phi)\in \cH^\alpha_{\gamma, L, N, G} \times P^\alpha(\vep) \mapsto \nu_{f,\phi}\in C^{\alpha}(M,\R)^*$.
	\item $(f,\phi)\in \cH^\alpha_{\gamma, L, N, G} \times P^\alpha(\vep) \mapsto \mu_{f,\phi}\in C^{\alpha}(M,\R)^*$.
\end{enumerate}
follow from 
\cite[Theorem A]{chaos}.  The fact that these maps vary analytic with respect to $\phi$ also follows from 
Theorem A in \cite{chaos}.

\end{proof}

\begin{proof}[Proof of Theorem~\ref{thm. spectralgapCr}]
For $(f,\phi)\in \cH^{r+\alpha}_{\gamma, L, N, G} \times P^r(\vep,\vep')$,
consider some $m$ given by Proposition~\ref{key proposition} and the corresponding 
$L_m=L^m$, $\gamma_m \in (0,1)$ and $q_m<\deg(f^m)$. Then we consider 
$L$ sufficiently close to 1 and $\vep, \vep'>0$ small enough  such that
we can apply   \cite[Theorem 2.1]{BCV16} (Theorem 2.1) and \cite[Theorem A]{chaos} 
for the map $g=f^m$ and the potential $\tilde \phi=\sum_{j=0}^{m-1}\phi\circ f^j$. 
Then we conclude that the transfer operator $\cL_{g,\tilde\phi}$ has the spectral gap property in the space of $C^r$ observables.
Denoting by $\hat\lambda_{g,\tilde\phi}=\rho(\mathcal{L}_{g,\tilde\phi})$ the spectral radius of $\mathcal{L}_{g,\tilde\phi}$, it follows that 
	there exists a unique probability measure $\hat\nu_{g,\tilde\phi}$  such that 
	$\mathcal{L}_{g,\tilde\phi}^*\hat\nu_{g,\tilde\phi}=\hat\lambda_{g,\tilde\phi} \hat\nu_{g,\tilde\phi}$, 
	there is a unique positive function $\hat h_{g,\tilde\phi}: M \to \R$ of class $C^r$ bounded from zero and infinity such that $\mathcal{L}_{g,\tilde\phi}\hat h_{g,\tilde\phi}=\hat\lambda_{g,\tilde\phi}\hat h_{g,\tilde\phi}$
	with $\int\hat h_{g,\tilde\phi}\, d\hat\nu_{g,\tilde\phi}=1$, 
	the $g$-invariant measure $\hat\mu_{g,\tilde\phi}=\hat h_{g,\tilde\phi}\hat\nu_{g,\tilde\phi}$ is the unique equilibrium state of $(g,\tilde\phi)$, it has exponential decay of correlation for $C^r$ observables and is exact.

The same argument used in the proof of Theorem~\ref{thm.spectralgap} guarantees the transfer operator $\cL_{f,\phi}$ has the spectral gap property in the space of $C^r$ observables and the properties (1), (2) and (3) above hold for $\cL_{f,\phi}$ replacing $h_{g,\tilde\phi}$, $\nu_{g,\tilde\phi}$ and $\lambda_{g,\tilde\phi}$ by
$h_{f,\phi}=h_{g,\tilde\phi}$, $\nu_{ f,\phi}=\nu_{g,\tilde\phi}$ and $\lambda_{f, \phi }=\sqrt[m]{\lambda_{g,\tilde\phi}}$, respectively.

Now using \cite{chaos} again, it follows that the following maps are  $C^{r-1}$:
\begin{enumerate}
	\item $(f,\phi)\in \cH^{r+\alpha}_{\gamma, L, N, G} \times P^r(\vep,\vep') \mapsto \lambda_{f,\phi}\in \R$.
	\item $(f,\phi) \in \cH^{r+\alpha}_{\gamma, L, N, G} \times P^r(\vep,\vep') \mapsto h_{f,\phi}\in C^0$;
	\item $(f,\phi)\in \cH^{r+\alpha}_{\gamma, L, N, G} \times P^r(\vep,\vep') \mapsto \nu_{f,\phi}\in C^{r}(M,\R)^*$.
	\item $(f,\phi)\in \cH^{r+\alpha}_{\gamma, L, N, G} \times P^r(\vep,\vep') \mapsto \mu_{f,\phi}\in C^r(M,\R)^*$.
\end{enumerate}
Moreover, these maps are analytic with respect to $\phi$.
\end{proof}

\subsection*{Differentiability of the mean and variance}

As consequence of the exponential decay of correlations, 
it follows  the Central Limit Theorem. We are mainly interested in the differentiability of the mean and the variance present in this theorem. 

\begin{proof}[Proof of Corollary \ref{c.CLT}]
	Part (i) follows as a consequence of the spectral gap property just as in \cite[Corollary C]{CV2013}. 
	
To part (ii) we use \cite[Theorem D]{BCV16} to conclude that the average 
	$$
	\cH^{r+\alpha}_{\gamma, L, N, G} \times P^r(\vep,\vep') \times C^r(M,\R) \ni 
	(f, \phi, \varphi) \to m_{\mu_{(f,\phi)}}(\varphi) := \int \varphi d\mu_{(f,\phi)} \in \R
	$$
	and  the variance 
	$$
	\cH^{r+\alpha}_{\gamma, L, N, G} \times P^r(\vep,\vep') \times C^r(M,\R) \ni 
	(f, \phi, \varphi) \to \sigma^2_{f,\varphi,\mu_{(f,\phi)}} \in \R
	$$
	are analytic with respect to $(\phi,\varphi)$.
	The proof that these maps are also $C^{r-1}$ with respect to $f$  follows similarly as in Theorem B of \cite{BC19}.
\end{proof}

\subsection*{Differentiability of the free energy}
\begin{proof}[Proof of Corollary \ref{C.free}]
	This proof is similar to the proof of \cite[Proposition 5.2]{BCV16}, we describe here the main steps.
	Observe that
	$$\int e^{t S_n\psi} d\mu_{f,\phi}  = \int \lambda_{f,\phi}^{-n} \mathcal{L}^n_{f,\phi}(h_{f,\phi} e^{tS_n\psi}) d\nu_{f,\phi} = \Big( \dfrac{\lambda_{f,\phi+t\psi }}{\lambda_{f,\phi}} \Big)^n \int \lambda_{f,\phi + t\psi}^{-n} \mathcal{L}^n_{f,\phi + t\psi}(h_{f,\phi}) d\nu_{f,\phi} $$ 
	
	Since $\lambda_{f,\phi + t\psi}^{-n} \mathcal{L}^n_{f,\phi+t\psi}(h_{f,\phi})$ converges uniformly to $h_{f,\phi + t\psi} \int h_{f,\phi} d\nu_{f,\phi + t\psi}$, by the   dominated convergence theorem:
	$$\mathcal{E}_{f,\phi,\psi}(t)=\lim_n \dfrac{1}{n} \log \int e^{t S_n\psi} d\mu_{f,\phi}  =  \log \lambda_{f,\phi+t\psi} - \log \lambda_{f,\phi} = P_{top}(f,\phi+t\psi) - P_{top}(f,\phi). $$
	
	If $\psi=\tilde\psi\circ f-\tilde\psi+A$ for some observable $\tilde\psi$ and constant $A$, then  $t \mapsto\mathcal{E}_{f,\phi,\psi}(t)$ is a linear map. In fact, we have
	$$
	\cL_{f,\phi+t\psi}(e^{t\tilde\psi}h_{f,\phi})=(e^{tA}\lambda_{f,\phi})e^{t\tilde\psi}h_{f,\phi}
	$$
	which implies that $\lambda_{f,\phi+t\psi}=e^{tA}\lambda_{f,\phi}$ and $h_{f,\phi+t\psi}=e^{t\tilde\psi}h_{f,\phi}$. In particular, $\mathcal{E}_{f,\phi,\psi}(t)=tA$. 
	When $\psi$ is not cohomologous to a constant we can prove that $\mathcal{E}''_{f,\phi,\psi}(t)>0$ for all $t$ small enough, and it follows the convexity.
	
\end{proof}

As a consequence, let us notice some properties of the free energy function:

\begin{corollary}
	For any Holder continuous potential $\psi$ so that $\int \psi d\mu_{f,\phi}=0$ the free energy function $[-t_{\phi,\psi}, t_{\phi,\psi}] \ni t \mapsto \mathcal{E}_{f,\phi,\psi}(t)$ satisfies:
	
	a)  $\mathcal{E}_{f,\phi,\psi}(0)=0$ and $\mathcal{E}_{f,\phi,\psi} (t) \geq 0 $ for all $t \in (-t_{\phi,\psi},t_{\phi,\psi})$;
	
	b) $t \inf \psi \leq \mathcal{E}_{f,\phi,\psi}(t) \leq t \sup \psi $ for all $t \in (0, t_{\phi,\psi})$;
	
	c) $t \sup \psi \leq \mathcal{E}_{f,\phi,\psi} \leq t \inf \psi$ for all $t \in (-t_{\phi,\psi},0)$.
\end{corollary}
These properties follows immediately from the convexity, using $\mathcal{E}'_{f,\phi,\psi}(0) = \int \psi d\mu_{f,\phi}=0$ and more generally
$\mathcal{E}'_{f,\phi,\psi}(t) = \int \psi d\mu_{f,\phi+t\psi}$ (what implies $\inf \psi \leq \mathcal{E}_{f,\phi,\psi}'(t) \leq \sup \psi$).

\subsection*{Differentiability of the large deviations rate functions}

When $\psi$ is not cohomologous to a constant and $\int \psi d\mu_{f,\phi}=0$, by convexity it follows that it is well defined the ``local" Legendre transform given by:

$$I_{f,\phi,\psi}(s) = \sup_{-t_{\phi,\psi} \leq t \leq t_{\phi,\psi} } \big( st - \mathcal{E}_{f,\phi,\psi}(t) \big)$$

Noticing that $\mathcal{E}_{f,\phi,\psi+c}(t) = \mathcal{E}_{f,\phi,\psi}(t) + ct$, we get that
$$I_{f,\phi,\psi+c}(t) = I_{f,\phi,\psi}(t-c).$$

When the free energy function is differentiable it follows the variational property
\begin{equation}
I_{f,\phi,\psi}( \mathcal{E}'_{f,\psi}(t) ) = t \mathcal{E}'_{f,\psi}(t) - \mathcal{E}_{f,\phi,\psi}(t)
\end{equation}
and the domain of $I_{f,\phi,\psi}$ contains the interval $[\mathcal{E}'_{f,\psi}(-t_{\phi,\psi}), \mathcal{E}'_{f,\psi}(t_{\phi,\psi}) ]$.
Moreover, $I_{f,\phi,\psi}(s)=0$ if and only if $s=m_{f,\phi}$ belongs to the domain of $I_{f,\phi,\psi}$.
Also, the convexity of $\mathcal{E}_{f,\phi,\psi}$ together with the differentiability of $\mathcal{E}_{f,\phi,\psi}$ gives that $t\mapsto I_{f,\phi,\psi}(t)$ is strictly convex and differentiable.

Using the differentiability of the free energy function, it follows the following result of large deviations by Gartner-Ellis’s theorem, with rate function given by the Legendre transform.

\begin{theorem*}
Consider $(f,\phi,\psi) \in 
\cH^{r+\alpha}_{\gamma, L, N, G} \times P^r(\vep,\vep')\times C^r(M,\R)$, $r \geq 2$, such that
 $\psi \in C^r$ is not cohomologous to a constant, then for any interval $[a,b] \subset [\mathcal{E}'_{f,\psi}(-t_{\phi,\psi}),   \mathcal{E}'_{f,\psi}(t_{\phi,\psi}) ] $ it holds that:
$$\underset{n \to \infty}{\limsup} \dfrac{1}{n} \log \mu_{f_v, \phi_v }\Big( x \in M, \dfrac{1}{n} S_n \psi_v (x) \in [a,b]  \Big)  \leq - \inf_{s \in [a,b]} I_{f_v, \phi_v, \psi_v}(s)$$
and  
$$\underset{n \to \infty}{\liminf} \dfrac{1}{n} \log \mu_{f_v, \phi_v }\Big( x \in M, \dfrac{1}{n} S_n \psi_v (x) \in (a,b)  \Big)  \geq - \inf_{s \in [a,b]} I_{f_v, \phi_v, \psi_v}(s).$$
\end{theorem*}

In the following we prove the regularity of the large deviations rate function with respect to the dynamics.

\begin{proof}[Proof of Corollary \ref{cor.I}]
	Let $\psi_{v*}$ be an observable which is not cohomologous to a constant, then there exists a neighborhood $U$ of $v_*$ such that every $\psi_{v*}$ is not cohomologous to a constant.
	
	For each $v\in U$, put $J_v=[\mathcal{E}'_{f_v,\psi_v}(-t_{\phi_v,\psi_v}), \mathcal{E}'_{f_v,\psi_v}(t_{\phi_v,\psi_v}) ]$ and consider $J$ an interval  contained in $\bigcap_{v\in U}J_v$.  
	
	Using the variational property of the Legendre transform and that
	$\mathcal{E}''_{f_v,\phi_v,\psi_v}(t) > 0$, we have that for all
	$ s \in  J$   
	there exists a unique $t=t(s,v)$ such that $s= \mathcal{E}'_{f_v,\phi_v,\psi_v}(t)$ and 
	\begin{equation}\label{eq}
	I_{f_v,\phi_v,\psi_v}(s) = st(s,v) - \mathcal{E}_{f_v,\phi_v,\psi_v}(t(s,v))
	\end{equation}
	
	By taking a smaller neighborhood $U$ of $v_*$ if necessary, there exists an open interval $J$ such that \eqref{eq} holds for all $s \in \overline{J}$ and $v \in U$. Now, consider
	\begin{align*}
	F: V \times \overline{J} &\to V \times \R  \\
	(v,t) &\mapsto (v, \mathcal{E}'_{f_v,\phi_v,\psi_v}(t) ).
	\end{align*}

	$F$ is injective because it is strictly decreasing along the fibers. As $V \times \overline{J}$ is a compact metric space, then $F$ is a homeomorphism onto its image. Actually, it is a $C^{r-1}$ diffeomorphism. Applying the Implicit Function Theorem to the equation
	\begin{equation}\label{implicit}
	F(v,t(v,s)) = (v,s)  \quad \quad \text{ where } \quad s= \mathcal{E}'_{f_v,\phi_v,\psi_v}(t),
	\end{equation}
	we have that for all  $(v,s) \in F(V \times \overline{J})$ there exists a unique $t=t(s,v)$ such that \eqref{implicit} is valid and $(v,s) \mapsto t(v,s)$ is $C^{r-1}$.
	
	From \eqref{eq}, it follows that $(v,s) \mapsto I_{f_v,\phi_v,\psi_v}(s)$ is $C^{r-1}$.
	
\end{proof}

\section{Examples}

In this Section we describe examples of dynamical systems that satisfy the assumption of Theorems A and B.

\subsection{Bad set covered with few domains of injectivity  (as in \cite{BCV16})}

The robust class  of maps in \cite{CV2013, BCV16} satisfy the condition (C), we describe them here.

Let $f:M \to M$ be a local homeomorphism and assume that there exists a function $x \mapsto L(x)$ such that for every $x \in M$ there exist a neighborhoud $U_x$ of $x$ so that the restriction $f_x: U_x \to f(U_x)$ is invertible and 
$$d(f_x^{-1}(y), f_x^{-1}(z) ) \leq L(x) d(y,z), \quad \text{ for all } y,z \in f(U_x).$$

Assume that there are constants $\sigma>1$ and $L\geq 1$  and an open region $A \subset M$ such that:

(H1) $L(x) \leq L$  for every $x \in A$ and $L(x) < \sigma^{-1}$ for all $ x \notin A$ and $L \in [1, L_0]$ for some $L_0$ close to $1$ (the precise condition is given in [BCV]).

(H2) There exists a finite covering $\mathcal{U}$ of $f$ by open domains of injectivity for $f$ such that $A$ can be covered by $q<\operatorname{deg}(f)$ elements of $\mathcal{U}$.

It is clear that every map $f$ satisfying (H1) and (H2) also satisfy (C),  since for every point $x\in M$ each pre-image lies at most in one set of $\mathcal{U}$, so at least one pre-image $y$ is not in  $A$, so the inverse branch that sends $x$ to $y$ is contractive.

So, if the potential has small variation and $L$ is close enough to $1$, then Theorems A and B are valid for this map.

\subsection{Derived from expanding}

Consider a $C^r$-expanding map $A: M \to M$ of degree $G$, $\mathcal{U}$ some covering of $M$ by domains of injectivity, $p\in M$ a fixed point of $A$ and $U \in \mathcal{U}$ one set containing $p$. Deform $A$ on a small neighbourhood $V \subset U$ of $p$ through a bifurcation in order to change de index of $p$. This can be done in order that either $p$ becomes a saddle or a sink.

Consider $f$ a map after the bifurcation, then $f$ coincides with $A$ outside $V$, so its action is  expanding outside of $V$. This implies that every point has a contracting inverse branch.
We consider $f$ so that $f$ is not too contracting in $V$, this guarantees that $f\in \mathcal{H}^r(\gamma,1,L,G)$ for some $L$ arbitrarily close to $1$.

For this map Theorem B is valid, so if $\phi$ is a potential $C^r$-close to the null potential, then $\mu_{f,\phi}$ is the unique equilibrium state for $f$ and depends smoothly with respect to $(f,\phi)$.

\begin{remark}
	It is particulary interesting  this construction when $p$ becomes a sink after the bifurcation. If $\{f_t\}$ denotes the parameterized family, supposing that it is uniformly expanding for $t<t_0$ and that $p$ is a sink for $t>t_0$, it happens that the equilibrium state will be supported in the whole torus for $t<t_0$ and in a proper subset for $t>t_0$, and the density of the equilibrium state will still vary smoothly in a neighborhood of the  parameter $t_0$.
\end{remark}

\subsection{Manneville-Pomeau like maps}
 Take the $C^{\infty}$ bump function 
 $$
 B(x)=\left\{\begin{array}{llr}
 	\dfrac{\exp(-x^{-1})}{\exp(-x^{-1})+\exp(-(1-x)^{-1})}, & \text{for } 0<x<1;\\
 	0, & \text{for } x\leq 0;\\
 	1, & \text{for } x\geq 1.
 \end{array}\right.
 $$
 
 \begin{figure}[h]
 	\centering
 	\includegraphics[width=0.7\linewidth]{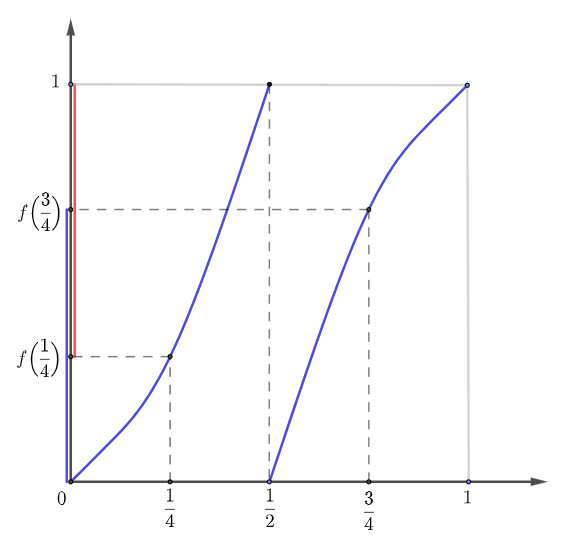}
 	\caption{Graph of Maneville-Pomeau like map $f$.}
 	\label{fig:pomeaulike}
 \end{figure}
  and consider $f: S^1 \to S^1$, where we identify $S^1$ with $\R/\Z$, that satisfies
\[
f(x) = \left\{\begin{array}{lr}
	x + B(x), & \text{for } 0\leq x\leq 1/2\\
	(x-1) + B(x), & \text{for } 1/2< x < 1
\end{array}\right. 
\]

Then $f$ is a $C^{\infty}$ endomorphism that satisfies $f(0)=0$, $f'(0)=1$ and $f'(x)>1$ if $x\neq 0$. Moreover it is easy to see that $f'$ is increasing in $[0,1/2]$ and decreasing in $[1/2,1]$ and therefore $f'(x)\geq f'(1/4)=f'(3/4)>1$ for every $x\in [1/4,3/4]$. Also $f([1/4,3/4])=S^1$ we conclude that $f$ satisfies the condition (C) in any topology $C^{r+\alpha}$, see Figure \ref{fig:pomeaulike} the graph of $f$.


Consider the family of potentials $\varphi_{t} = -t \log |Df|$. It is clear that $\phi\in P^{r+\alpha}(\vep,\vep')$  for $|t|$ small enough. So Theorem B is valid for these maps.
For $t=0$, the corresponding equilibrium state is the maximum entropy measure.

One interesting application is if we perturb $f$ in order to create a sink at $x=0$, so the equilibrium state of the perturbed dynamical system will be supported in a Cantor set not containing $0$. Since Theorem B is valid for a robust class, the equilibrium state supported in the Cantor set will be $C^{r-1}$-close to the measure of maximum entropy of the original system.

\subsection{Hyperbolic potentials for topologically exact local homeomorphism}

The dynamical systems considered in this paper includes the local homeomorphisms with hyperbolic potentials considered in \cite{ARS2021}, let us describe them.

Assume that $f:M \to M$ is a local homeomorphism, $\phi$ a $\sigma$-hyperbolic potential Holder continuous.
Consider $\{U_1, \cdots, U_m\}$ a finite open cover by domains of injectivity and $\nu = \max_j\{ \sup_{x \in U_j} \|Df^{-1}(x)\} $.
Assuming that $f$ is topologically exact (i.e., for every open set $U$ there exists $N$ such that $f^N(U)=M$), that $\phi$ has small variation and
$$\nu \cdot \sigma <1,$$
then \cite{ARS2021} proves that there exists a unique equilibrium state.

Actually, in the proof of Lemma 3.2 in \cite{ARS2021} they verify that there exists some $N$ so that for every $x \in M$ there exists an inverse branch of $f^N$ that contracts at some rate $\gamma$, which means that condition (C) is valid.

We remark that the present work extends the result of \cite{ARS2021} without assuming that $f$ is topologically exact.

\subsection{Every point with some contracting inverse branch}

Although condition (C) asks that the integer $N$ is the same for every $x\in M$, we have that the same condition is valid if we ask that for every $x\in M$ there exist some pointwise $n=n(x)$ and some inverse branch $h_n$ of $T^n$ so that $h_n$ is a contraction.

\begin{proposition}
	Suppose that $f:M \to M$ is a local homeomorphism and that for every $x\in M$ there exist some $n=n(x)$ and some inverse branch $h_n$ of $f^n$ so that $h_n$ is a contraction. Then it is valid condition (C), that is, there exists $\tilde N$ so that for every $x \in M$ there exists some inverse branch $h$ of $f^{\tilde N}$ so that $h$ is a contraction.
\end{proposition}
\begin{proof} 
		By continuity, every $x$ admits a neighborhood $U_x$ such that $h_{n(x)}$ contracts at rate $\gamma_x<1$ in $U_x$.
	By compactness we can cover $M$ by finite neighborhoods $U_j:=U_{x_j}$, $j=1,\cdots, r$, in which $h_{n_j}$ is an inverse branch that contracts at rate $\gamma := \max \gamma_{x_j}$.
	
	Take $N=\max_j n_j$ and $\kappa$ an integer such that $\gamma^\kappa L^N<1$, where $L=\sup L_f(x)$. Also take $\delta$ the Lebesgue number of the cover above. So given $x\in M$, the ball centered in $x$ and radius $\delta$ is contained in some $V_1:=U_{j_1}$ and then, $g_1$ as the restriction of $h_{j_1}$ to the this ball, it is a $\gamma$-contracted inverse branch of $f^{n_{j_1}}$. Define inductively $y_1=x$, $y_2=g_1(y_1)$, $\cdots$, $y_{s+1}=g_{s}(y_{s})$ such that the ball centered in $y_i$ and radius $\delta$ is contained in some $V_i:=U_{j_i}$, where $g_i$ is the restriction of $h_{j_i}$ to the ball $B(y_i,\delta)$. $g_i$ is therefore a $\gamma$-contracted inverse branch of $f^{n_{j_i}}$. Moreover, we choice $s\geq \kappa$ such that $(\kappa-1)N<n_{j_1}+\cdots n_{j_s}\leq \kappa N$, let $m=\kappa N-(n_{j_1}+\cdots n_{j_s})$, then for any inverse branch $g$ of $f^m$ around $x_s$ we have that $h=g\circ g_s\circ g_{s-1} \circ \cdots \circ g_1$ is an inverse branch of $f^{\tilde N}$ (with $\tilde N=\kappa N$) that contracts at least rate $\gamma^{\kappa} L^N$.
\end{proof}

\end{document}